\documentclass[12pt]{amsart}
\usepackage[normalem]{ulem}
\usepackage[usenames]{color}
\usepackage{graphicx,psfrag}

\DeclareMathOperator{\spn}{span}
 \newcommand{\R}{\mathbb R}\newcommand{\A}{\mathbf a}\newcommand{\B}{\mathbf b}
\newcommand{\rec}{\mathcal{R}} \newcommand{\Z}{\mathbb Z}
\newcommand{\ep}{\varepsilon}

\DeclareMathOperator{\Int}{Int}
\DeclareMathOperator{\cl}{cl}

\newcommand{\ds}{\displaystyle}
\newcommand{\crt}{r_\ep}
\newcommand{\mt}{m_\ep}
\newcommand{\nspan}{N_{\spn}}
\DeclareMathOperator{\diam}{diam}
\newcommand{\lip}{c}

\theoremstyle{plain} \newtheorem{thm}{Theorem}
\newtheorem{cor}[thm]{Corollary} \newtheorem{prop}[thm]{Proposition}
\newtheorem{lemma}[thm]{Lemma}

\theoremstyle{definition} \newtheorem{defn}[thm]{Definition}

\newtheorem{ex}[thm]{Example} 

\theoremstyle{remark} \newtheorem{remark}[thm]{Remark}

\begin{document}

\title{Chain recurrence rates and topological entropy}

\author{David Richeson}   \author{Jim Wiseman} \address{Dickinson
College\\ Carlisle, PA 17013} \email{richesod@dickinson.edu} 
\address{Agnes Scott College \\ Decatur, GA 30030}
\email{jwiseman@agnesscott.edu}

\begin{abstract}
We investigate the properties of chain recurrent, chain transitive, and chain mixing maps (generalizations of the well-known notions of non-wandering, topologically transitive, and topologically mixing maps).  We describe the structure of chain transitive maps.  These notions of recurrence are defined using $\ep$-chains, and the minimal lengths of these $\ep$-chains give a way to measure recurrence time (chain recurrence and chain mixing times). We give upper and lower bounds for these recurrence times and relate the chain mixing time to topological entropy.
\end{abstract}

\maketitle

\section{Introduction}Pseudo-orbits, or $\ep$-chains, are important tools for investigating properties of discrete dynamical systems.  These $\ep$-chains detect recurrent and mixing behaviors that may not be evident by studying actual orbits.  However, what is missing from the definitions of terms such as chain recurrence, chain transitivity, and chain mixing is any information about the lengths of the recurrence or mixing times.  They also do not give fine detail about the ``dynamics'' of the $\ep$-chains.

Let $X$ be a compact metric space and $f:X\to X$ be a continuous map.  The map $f$ is chain recurrent if for every $\ep>0$ and every point $x$, there is an $\ep$-chain from $x$ to itself (see section~\ref{sec:defns} for definitions).  The chain recurrence time (that is, the length of the shortest such chain) depends on $\ep$.  Similarly, if $f$ is chain mixing, then the chain mixing time also depends on $\ep$.  We discuss the properties and relationship of these two quantities, and in particular relate the chain mixing time to topological entropy.

We also study the structure of chain transitive maps, which resembles that of the symbolic dynamics determined by irreducible graphs and is rather different from that of topologically transitive maps.  For example, an irrational rotation of a circle is topologically transitive (actually minimal), but certainly not topologically mixing.  However, as we will see (Cor.~\ref{cor:connected}), a map on a connected space is chain mixing if and only if it is chain transitive (and, in fact, if and only if it is chain recurrent).

The paper is organized as follows.  In section~\ref{sec:defns} we give definitions.  In section~\ref{sec:structure} we discuss the structure of chain transitive maps.  In section~\ref{sec:estimates} we explore the properties of the chain recurrence and mixing times, and in section~\ref{sec:entropy} we relate the chain mixing time to topological entropy.  In section~\ref{sec:examples} we study examples.

\section{Definitions}
\label{sec:defns}

Let $X$ be a compact metric space. (To avoid technical difficulties later, we assume that $X$ is not a single point.) Throughout the paper, let $f:X\to X$ be a continuous map.  An \emph{($\ep,f$)-chain (or ($\ep,f$)-pseudo-orbit)} from $x$ to $y$ is a sequence $(x=x_0, x_1, \dots, x_n=y)$ such that $d(f(x_{i-1}),x_i)\le\ep$ for $i=1,\dots,n$.   To simplify notation, in most cases we write simply ``$\ep$-chain'' without referring explicitly to the map $f$. We define the \emph{length} of the $\ep$-chain $(x_0, x_1, \dots, x_n)$ to be $n$.

There are a variety of ways to use $\ep$-chains to describe recurrence.  In this paper we investigate three such notions---chain recurrence, chain transitivity, and chain mixing.  They generalize the well-known topological notions of recurrence of non-wandering, topological transitivity, and topological mixing.   

For the sake of comparison, let us recall the definitions of these topological notions of recurrence.  A point $x\in X$ is \emph{non-wandering} if for any open neighborhood $U\subset X$ of $x$, there exists $n>0$ such that $f^{n}(U)\cap U\ne \emptyset$.  The map $f$ is non-wandering if every point of $X$ is non-wandering. A map $f$ is \emph{topologically transitive} if for any nonempty open sets $U, V\subset X$, there is an $n\ge 0$ such that $f^{n}(U)\cap V\ne\emptyset$.   A map $f$ is \emph{topologically mixing} if for any nonempty open sets $U,V\subset X$ there is an $N>0$ such that $f^n(U)\cap V\ne\emptyset$ for all $n\ge N$.

\emph{Chain recurrence.} A point $x$ is \emph{chain recurrent} if for every $\ep>0$, there is an $\ep$-chain from $x$ to itself.  The map $f$ is \emph{chain recurrent} if every point of $X$ is chain recurrent.

If $x$ is a chain recurrent point, define the \emph{$\ep$-chain recurrence time} $\crt(x,f)$ to be the smallest $n$ such that there is an $\ep$-chain of length $n$ from $x$ to itself.  If $f$ is chain recurrent, define $\crt(f)$ to be the maximum over all $x$ of $\crt(x,f)$.  (To see that this maximum exists, observe that if there is an $\frac\ep2$-chain of length $n$ from $x$ to itself, then there is a neighborhood $U$ of $x$ such that for all $y\in U$, there is an $\ep$-chain of length $n$ from $y$ to itself.  Then the compactness of $X$ gives an upper bound on $\crt$.)  For the sake of brevity, we usually write ``$\crt(x)$'' and ``$\crt$.''

\emph{Chain transitivity.}  A map $f$ is \emph{chain transitive} if for every $x,y\in X$ and every $\ep>0$, there is an $\ep$-chain from $x$ to $y$.  Following \cite{GF} we say that a map $f$ is \emph{totally chain transitive} if $f^n$ is chain transitive for all $n\ge0$.

\emph{Chain mixing.} 
By compactness, a map $f$ is topologically mixing if and only if for any $\ep,\delta>0$, there is an $N$ such that for any $x,y\in X$ and any $n\ge N$, there are points $x'\in B_\delta(x)$ and $y'\in B_\ep(y)$ with $f^n(x')=y'$.  To mimic this, we say that
$f$ is \emph{$\ep$-chain mixing} if there is an $N>0$ such that for any $x,y\in X$ and any $n\ge N$, there is an $\ep$-chain from $x$ to $y$ of length exactly $n$. The map $f$ is \emph{chain mixing} if it is $\ep$-chain mixing for every $\ep>0$.  
If $0<\ep<\delta$ and $x\in X$, define the \emph{chain mixing time} $\mt(x,\delta,f)$ to be the smallest $N$ such that for any $n\ge N$ and any $y\in X$, there is an $\ep$-chain of length exactly $n$ from some point in $B_\delta(x)$ to $y$.  (In other words, $\mt(x,\delta,f)$ is the smallest $N$ such that $(B_\ep\circ f)^N(B_\delta(x))=X$.)  We define $\mt(\delta,f)$ to be the maximum over all $x$ of $\mt(x,\delta,f)$.  (This maximum exists by compactness.)  Again, to simplify notation, in most cases we write ``$\mt(x,\delta)$'' and ``$\mt(\delta)$.''

If $f$ is a physical system, then $\delta$ represents the uncertainty in the measurement of the initial state of the system, and $\ep$ the possible error in evaluating the map at each step.  The chain mixing time $\mt(\delta)$ is the time after which all information about the initial state of the system is lost, since any point could be anywhere by then.

We are indebted to the referee for some extremely helpful comments and ideas.

\section{Structure of chain transitive maps}
\label{sec:structure}



In this section we investigate the structure of chain transitive dynamical systems.  To the best of our knowledge, the main result about this structure, Theorem~\ref{thm:structure}, appears only in the exercises of \cite{A}.  Thus a self-contained discussion of it and its consequences may be useful.

We show that for a fixed $\ep>0$, the space $X$ decomposes into finitely many subsets; $f$ permutes these $n$ subsets cyclically and $f^{n}$ is $\ep$-chain mixing on each of them.  At one extreme, if $\ep$ is greater than the diameter of $X$, $f$ is $\ep$-chain mixing on $X$ itself.  As $\ep$ goes to zero, the subsets divide repeatedly, either finitely or infinitely many times.

For a fixed $\ep>0$, the dynamical structure of $\ep$-chains for a chain transitive map is very similar to that of the symbolic dynamics determined by a finite irreducible graph.  (See \cite[\S4.5]{LM}, from which we borrow some terminology.)  As $\ep$ goes to zero, the size of the graph may grow, possibly without bound.

Let us consider two elementary examples.

\begin{ex}
Let $X=S^1$ (considered as $\R/\Z$).  The doubling map $x\mapsto 2x$ is chain mixing.  
\end{ex}

\begin{ex}\label{ex:doublecircle}
Let $X$ be the disjoint union of two circles and $f$ be the map sending a point $x$ to the point $2x$ in the other circle.  The map $f$ is chain transitive but not chain mixing, because it is not $\ep$-chain mixing for any $\ep$ smaller than the distance between the two circles.  However, $f^2$ restricted to one circle \emph{is} chain mixing.
\end{ex}

We should not be misled by these examples;  they are not prototypes for all chain transitive maps. As $\ep$ decreases the number of permuted subsets may go to infinity.  As we see in the following example, it is possible that no power of $f$ has associated invariant subsets on which it is chain mixing. 

\begin{ex}
Let $X$ be the Cantor set of sequences $\A=(a_1, a_2,\ldots)$, where each $a_i\in \left\{ 0,1\right\}$.  Give $X$ the metric $d(\A,\B)=  \sum_{i=1}^\infty \delta(a_i,b_i)/2^i$, where $\delta(a_i,b_i)=0$ if $a_i=b_i$ and 1 otherwise.  Define  $f:X\to X$ by adding $(1,0,0,0,\ldots)$, where addition is defined componentwise base 2, with carrying to the right.  (For example, $f(1,0,0,0,\ldots)= (1,0,0,0,\ldots) + (1,0,0,0,\ldots) = (0,1,0,0,\ldots)$, and $f(1,1,1,0,0,\ldots)=(0,0,0,1,0,\ldots)$.)  The map $f$, which is minimal, is called an \emph{adding machine} or \emph{odometer map}.

If we look only at the first digit, then $f$ has period two.  On the first two digits, $f$ has period four ($(0,0) \mapsto (1,0) \mapsto (0,1) \mapsto (1,1) \mapsto (0,0)$), and so on.   Thus $f$ is $\ep$-chain mixing for $\ep\ge\frac{1}{2}$, but not for $\ep<\frac{1}{2}$.  However, when $\frac14\le\ep<\frac12$, $f^2$ restricted to $\{(0,a_2,a_3,\ldots)\}$ or $\{(1,a_2,a_3,\ldots)\}$ is $\ep$-chain mixing, and in general, for $\frac1{2^{k}}\le\ep<\frac1{2^{k-1}}$,  $X$ decomposes into $2^{k}$ disjoint sets on which $f^{2^k}$ is $\ep$-chain mixing, but $f^{2^{k-1}}$ is not.
\end{ex}

In fact, this is essentially the only example of a chain transitive map $f$ such that no $f^k$ is a disjoint union of finitely many chain mixing maps, in the sense that any such map factors onto an adding machine map, which we now define more generally.  (Adding machines arise in many dynamical contexts; see, for example, \cite{AHK,Bl,BS,CPY,Ku,MM,N}.  See \cite{BK} for more detailed definitions and background.)

\begin{defn}[\cite{BK}]
Let $J=(j_1, j_2,\ldots)$ be a sequence of integers greater than or equal to 2.  Let $X_J$ be the Cantor set of all sequences $(a_1, a_2,\ldots)$ where $a_i\in\{0,1,\dots,j_i-1\}$ for all $i$, with $d(\A,\B)= \sum_{i=1}^\infty \delta(a_i,b_i)/2^i$.  Define the \emph{adding machine map} $f_J:X_J\to X_J$ by $f_J(a_1, a_2,\ldots) = (a_1, a_2,\ldots) + (1,0,0,\ldots)$, where addition is defined componentwise mod $j_i$, with carrying to the right.  
\end{defn}

It is easy to see that every adding machine $f_J$ is minimal.  The following alternative definition of an adding machine as an inverse limit of cyclic groups will be useful.

\begin{lemma}[{\cite[p.\ 21]{GS}}]\label{lem:eqdefn}
Let $m_1,m_2,\dots$ be an increasing sequence of positive integers such that $m_i$ divides $m_{i+1}$ for all $i$.
Let $p_{m_i}:\Z/m_i\to \Z/m_i$ be the cyclic permutation given by $p_{m_i}(n)=n+1 \mod m_i$.  Then the inverse limit map $p(\{n_i\})=\{n_i+1\mod m_i\}$ is topologically conjugate to the adding machine map $f_J$, where $J=(m_1, m_2/m_1, m_3/m_2,\dots)$.
\end{lemma}

The following result is essentially contained in exercises 8.22 and 9.18 of \cite{A}.  We include it here for completeness and to fill in the gaps in the proofs sketched there.

\begin{thm}[\cite{A}]
\label{thm:structure}
Let $f:X\to X$ be a chain transitive map.  Then either
	\begin{enumerate}
	\item There is a period $k\ge1$ such that $f$ cyclically permutes $k$ closed and open equivalence classes of $X$, and $f^k$ restricted to each equivalence class is chain mixing; or
	\item $f$ factors onto an adding machine map.
	\end{enumerate}
\end{thm}

The general idea for the proof is to define an equivalence relation $\sim_{\ep}$ on $X$.  The induced map on the quotient space $X/\sim_{\ep}$ is a periodic orbit.  Sending $\ep$ to zero yields an equivalence relation $\sim$; thus $f$ factors onto the inverse limit of the periodic orbits, which will be either a periodic orbit or an adding machine.

\begin{lemma} Let $f:X\to X$ be a chain transitive map and let $\ep>0$.  There exists $k_{\ep}\ge 1$ such that for any $x\in X$, $k_{\ep}$ is the greatest common divisor of the lengths of all $\ep$-chains from $x$ to itself.
\end{lemma}
\begin{proof}
For each $x\in X$ define $k_\ep(x)$ to be the greatest common divisor of the lengths of $\ep$-chains from $x$ to itself.  Consider two points $x,y\in X$. Let $(y_0=y, y_1,\ldots, y_n=y)$ be any $\ep$-chain from $y$ to itself.  We show that $k_\ep(x)$ divides $n$.  Let $(x_0=x, x_1,\ldots,y,\ldots, x_{mk_\ep(x)}=x)$ be an $\ep$-chain from $x$ to $y$ and back to $x$.  The $\ep$-chain $(x_0=x, x_1,\ldots,y, y_1,\ldots,y_n=y,\ldots, x_{mk_\ep(x)}=x)$ from $x$ to itself has length $mk_\ep(x)+n$.  Since $mk_\ep(x)+n$ is necessarily a multiple of $k_\ep(x)$, $n$ is as well.
\end{proof}

Define a relation on $X$ by setting $x\sim_\ep y$ if there is an $\ep$-chain from $x$ to $y$ of length a multiple of $k_\ep$.  By concatenating chains, it is easy to see that this is an equivalence relation.  In fact, by the definition of $k_\ep$, if $x\sim_\ep y$, then \emph{any} $\ep$-chain from $x$ to $y$ must have length a multiple of $k_\ep$.  We may now define another equivalence relation on $X$ by saying $x\sim y$ if $x\sim_\ep y$ for all $\ep>0$.  

\begin{remark}Since $f$ is chain transitive, it is easy to see that $x\sim_\ep y$ if and only if there is a point $z$ such that there are $\ep$-chains from $x$ to $z$ and from $y$ to $z$ of the same length. Thus another way to define the equivalence relation $\sim$ is as the set of pairs $(x,y)$ such that the set of points reachable from $(x,y)$ with $(f\times f)$-chains (${\mathcal C}(f\times f)(x,y)$, in the notation of \cite{A}) meets the diagonal.
\end{remark}

\begin{lemma} Let $f:X\to X$ be a chain transitive map.  For any $\ep>0$ the equivalence relation $\sim_{\ep}$ is both open and closed.  The equivalence relation $\sim$ is closed.
\end{lemma}
\begin{proof}
It is clear that the relations $\sim_{\ep}$ and $\sim$ are closed.  To see that $\sim_{\ep}$ is open, observe that an $\frac\ep2$-chain is a fortiori an $\ep$-chain.  Since $f$ is chain transitive, there is an $\frac\ep2$-chain from $x$ to $y$ for any points $x$ and $y$, and, if $x\sim_\ep y$, the length of the chain must be a multiple of $k_\ep$.  Thus there is an $\ep$-chain of the same length from $x$ to any point within $\frac\ep2$ of $y$ (use the same $\frac\ep2$ chain and change only the last element).
\end{proof}

Now, let $\tilde K=X/\sim$ be the quotient space and $\tilde f:\tilde K\to \tilde K$ be the map induced by $f$.

\begin{proof}[Proof of Theorem \ref{thm:structure}]
To begin, let $\ep>0$ be fixed. There are $k_\ep$ equivalence classes for $\sim_{\ep}$, $f$ cycles among the classes periodically, and each class is invariant under $f^k$.  The quantity $k_\ep$ is nondecreasing as $\ep\to0$, and in fact $k_{\ep_2}$ divides $k_{\ep_1}$ if $\ep_1\le\ep_2$.  Either $k_\ep$ stabilizes at some $k$, or it grows without bound.  We consider the two cases separately.

If $k_\ep$ stabilizes at $k=1$, then there is only one $\sim$ equivalence class, and $f$ is chain mixing.  To see this, observe that for any $\ep$ and any $x$, since 1 is the GCD of the lengths of $\ep$-chains from $x$ to itself, there exist $\ep$-chains from $x$ to itself of lengths $m$ and $n$ with $m$ and $n$ relatively prime.  Sylvester's formula (\cite{Ra}) tell us that by concatenating copies of these loops, we can get an $\ep$-chain from $x$ to itself of length $N$ for any $N>mn-m-n$.  By compactness, there is an $M>0$ such that between any two points in $X$ there is an $\ep$-chain of length less than or equal to $M$.  By adding a loop at $x$ to a chain from $x$ to $y$, we can get a chain from $x$ to $y$ of any length greater than $M+N$, so $f$ is chain mixing.

More generally, if $k=k_{\ep_0}$, then the equivalence relation $\sim$ is the same as $\sim_{\ep_0}$.   Thus there are $k$ equivalence classes, $f$ cycles among the classes periodically, and each class is invariant under $f^k$. An argument similar to that for $k=1$ shows that $f^k$ is chain mixing on each equivalence class.

If, on the other hand, $k_\ep$ grows without bound, then the period of $f$'s cycling goes to infinity as $\ep$ shrinks.  In this case, $f$ factors onto an adding machine map.  More precisely, 
let $\tilde K_\ep$ be the quotient space $X$ modulo $\sim_\ep$, with the quotient topology, and let $\tilde f_\ep:\tilde K_\ep\to \tilde K_\ep$ be the induced map.  Let $\tilde f:\tilde K\to \tilde K$ be the induced map on the quotient space $\tilde K=X/\sim$. The map $\tilde f_\ep:\tilde K_\ep\to \tilde K_\ep$ is conjugate to the cyclic permutation $p_{k_\ep}:\Z/k_\ep\Z\to \Z/k_\ep\Z$ given by $p_{k_\ep}(n)=n+1 \mod k_\ep$.  Thus, if we take a sequence $\{\ep_i\}$ decreasing to 0 such that the sequence $\{k_{\ep_i}\}$ is strictly increasing, then $\tilde f$ is conjugate to the map $g:Z\to Z$, where $Z$ is the inverse limit of the sequence $\{p_{k_{\ep_i}}:\Z/k_{\ep_i}\Z\to \Z/k_{\ep_i}\Z\}$ and $g$ is given by $g(\{x_i\})=\{x_i+1 \mod k_{\ep_i}\}$.  Then, by Lemma~\ref{lem:eqdefn}, $g$ is topologically conjugate to the adding machine map $f_J$, where $J=(k_{\ep_1}, k_{\ep_2}/k_{\ep_1},  k_{\ep_3}/k_{\ep_2},\dots)$.  Thus the conjugacy class of $f_J$ is independent of the choice of the sequence $\{\ep_i\}$.  In fact, it is shown in \cite{BS} that the conjugacy class of an adding machine $f_{(j_1,j_2,\ldots)}$ depends only on the sums of the powers to which prime numbers occur as factors of the $j_i$'s.
\end{proof}
	
We have a number of nice corollaries that follow from Theorem \ref{thm:structure}.
	
\begin{cor}
Let $f:X\to X$ be a chain transitive map.  If $f$ has a point of period $n$, then there is a $k$ dividing $n$ such that $X$ decomposes into $k$ closed and open equivalence classes, $f$ permutes the equivalence classes cyclically, and $f^k$ restricted to each equivalence class is chain mixing.
\end{cor}

\begin{proof}
If $x\in X$ has period $n$ under $f$, then the image of $x$ in $\tilde K$ has period $k$ under $\tilde f$, for some $k$ dividing $n$.  Thus $\tilde f$ cannot be an adding machine map.
\end{proof}

\begin{cor}
Let $f:X\to X$ be a chain transitive map.  If one of the $\sim$-equivalence classes is open, then there is a $k$ such that $X$ decomposes into $k$ closed and open equivalence classes, $f$ permutes the equivalence classes cyclically, and $f^k$ restricted to each equivalence class is chain mixing.
\end{cor}

\begin{proof}
If one of the $\sim$-equivalence classes is open, then its projection in $\tilde K$ is a single point that is both open and closed.  Thus $\tilde K$ cannot be a Cantor set.
\end{proof}

By definition, if $f$ is chain mixing, then it is totally chain transitive.  By Theorem \ref{thm:structure}, the converse is also true.

\begin{cor}
$f:X\to X$ is totally chain transitive if and only if it is chain mixing.
\end{cor}

In fact, as is shown in \cite{Y}, these two conditions are also equivalent to chain transitivity for $f\times f$.
To put this result in perspective, we contrast it with topologically transitive maps.  (See also \cite{Ba}.)

\begin{ex}
An irrational rotation $f:S^{1}\to S^{1}$ is totally topologically transitive, but not topologically mixing.  

\end{ex}

Conley's fundamental theorem of dynamical systems (\cite[\S9.1]{R}), says that the interesting dynamics of $f$ takes place on the chain recurrent set $\rec$, with all other orbits simply going from one piece of $\rec$ to another.  In some special cases (if $\rec$ has a hyperbolic structure, for example (\cite[\S9.5]{R})), $\rec$ decomposes into finitely many chain transitive pieces (basic sets), so there is not much difference between chain recurrence and chain transitivity.  In general, the situation is more complicated, but we do have the following result.

\begin{cor}\label{cor:connected}
Let $X$ be connected and $f:X\to X$ be continuous.  Then the following are equivalent:
\begin{enumerate}
\item $f$ is chain recurrent.
\item $f$ is chain transitive.
\item $f$ is totally chain transitive.
\item $f$ is chain mixing.
\end{enumerate}
\end{cor}

\begin{proof}

Clearly (4) $\Rightarrow$  (3) $\Rightarrow$ (2) $\Rightarrow$ (1). By Theorem  \ref{thm:structure}, connectivity and chain transitivity imply that $f$ is chain mixing.  So it is enough to show that chain recurrence implies chain transitivity.
This follows from the proof of Proposition 1.1 in \cite{AkinOxtoby}, but for completeness we now give a more direct proof.
Assume that $f$ is chain recurrent and pick $\ep>0$.  We say that $x$ and $y$ are $\ep$-chain equivalent if there are $\ep$-chains from $x$ to $y$ and from $y$ to $x$.  Since $f$ is chain recurrent, this is an equivalence relation.  So to finish the proof it suffices to show that this is an open condition.    Let $x$ and $y$ be $\ep$-chain equivalent. Choose $\delta\le\ep/2$ such that if $d(y,y')<\delta$, then $d(f(y),f(y'))<\ep/2$.  It suffices to show that $x$ is $\ep$-chain equivalent to an arbitrary $z\in B_{\delta}(y)$.  Let $(x_0=x,x_1,\ldots, x_n=y)$ be an $\ep$-chain and let $(z_0=y, z_{1},\ldots, z_{m}=y)$ be an $\ep/2$-chain from $y$ to itself.  Then $(x_0=x,x_1,\ldots, x_n=y=z_{0},\ldots, z_{m-1},z)$ is an $\ep$-chain from $x$ to $z$.  Similarly, let $(y_0=y,y_1,\ldots,y_r=x)$ be an $\ep$-chain from $y$ to $x$.  Then  $(z, z_{1},\ldots, z_{m}=y=y_{0},\ldots, y_n=x)$ is an $\ep$-chain from $z$ to $x$.  Thus $x$ is $\ep$-chain equivalent to $z$.
\end{proof}

Recall (from \cite{Aus}, for example) the definition $R_\pi=\{(x,y): \pi(x)=\pi(y)\}$ for a factor map $\pi:(X,f)\to(Y,g)$.  The system $(X,f)$ is \emph{weak mixing} if $f\times f$ is topologically transitive; the factor map $\pi$ is \emph{weak mixing} if the restriction of $f\times f$ to $R_\pi$ is topologically transitive.  Analogously, we will define the factor map $\pi$ to be {\em chain mixing} if the restriction of $f\times f$ to $R_\pi$ is chain transitive.  The following proposition shows that when $\pi$ is the quotient map associated to $\sim$, then $\pi$ is chain mixing.

\begin{prop}
The relation $R_\pi=\sim$ is a maximal closed, chain transitive, invariant subset of $X\times X$.
\end{prop}

\begin{proof}
It is clearly closed and invariant.  To see that it is chain transitive, fix $\ep>0$ and points $x_0,y_0,x_1,y_1\in X$ with $x_{0}\sim y_{0}$ and $x_{1}\sim y_{1}$.  We must show that there is an $\ep$-chain for $f\times f$ from $(x_0,y_0)$ to $(x_1,y_1)$.  Let $\tilde K_0$ be the $\sim_\ep$-equivalence class containing $x_0$ and $y_0$, and $\tilde K_1$ be the $\sim_\ep$-equivalence class containing $x_1$ and $y_1$.  There is an $n$ such that $f^n$ maps $K_0$ onto $K_1$.  Since $f^{k_\ep}$ restricted to $\tilde K_0$ is chain mixing, there is an $\ep$-chain for $f\times f$ from $(x_0,y_0)$ to $(f^{-n}(x_1), f^{-n}(y_1)$, which we can extend by an actual orbit to  $(x_1,y_1)$.

To see that $\sim$ is a maximal chain transitive subset, observe that if $x_0\sim_\ep y_0$ and $(x_0,y_0), (x_1,y_1),\dots,(x_m,y_m)$ is an $\ep$-chain, then $x_i\sim_\ep y_i$ for all $i$.  Thus there cannot be an $\ep$-chain from $(x_0,y_0)$ to $(x,y)$ unless $x\sim_\ep y$.

\end{proof}
	
We conclude this section with a few results from the literature that are similar to Theorem \ref{thm:structure}.  However, instead of looking at the induced map on the space of equivalence classes $\tilde K$, these relate to the induced map on $K$, the set of connected components of $X$.  See  \cite{BS,HH,MDG} for details and more general results. The first is a  ``folklore'' result similar to Theorem \ref{thm:structure} for topologically transitive maps.

\begin{thm}[\cite{BS}]
Let $X$ be a compact metric space, $f:X\to X$ be topologically transitive map, and $\bar f:K\to K$ be the induced map on the connected components of $X$.  Then either
\begin{enumerate}
\item $K$ is finite and $\bar f$ is a cyclic permutation, or
\item $K$ is a Cantor set and $\bar f$ is topologically transitive.
\end{enumerate}
\end{thm}

Notice that under these hypotheses if $x$ and $y$ are in the same connected component, then necessarily $x\sim y$.  Thus $\bar f:K\to K$ factors onto $\tilde f:\tilde K\to \tilde K$.  If $X$ has only finitely many connected components, then the two decompositions and the two maps are the same.

\begin{defn}  Let $A\subset X$ be a compact invariant set.
\begin{enumerate}
\item $A$ is topologically transitive (resp.\ chain transitive) if $f|_A:A\to A$ is topologically transitive (resp.\ chain transitive).
\item $A$ is \emph{Liapunov stable} if every neighborhood $U$ of $A$ contains a neighborhood $V$ of $A$ such that $f^n(V)\subset U$ for all $n\ge0$.
\item $A$ is an \emph{attractor} if there exists a neighborhood $U$ of $A$ such that $\cl(f(U))\subset\Int U$ and $A=\bigcap_{n\ge0}f^n(U)$.
\end{enumerate}
\end{defn}

\begin{thm}[\cite{BS,HH}]
Let $X$ be a locally compact, locally connected metric space and $f:X\to X$ be a continuous map.  Let $A\subset X$ be a compact, Liapunov stable transitive set, and $\bar f:K\to K$ be the induced map on the connected components of $A$.  Then either
\begin{enumerate}
\item $K$ is finite and $\bar f$ is a cyclic permutation, or
\item $K$ is a Cantor set and $\bar f$ is topologically conjugate to an adding machine map.
\end{enumerate}
\end{thm}

\begin{thm}[\cite{HH}]\label{thm:HH}
Let $X$ be a locally compact, locally connected metric space and $f:X\to X$ be a continuous map.   Let $A\subset X$ be a compact, chain transitive attractor.  Then $\bar f:K\to K$, the induced map on the connected components of $A$, is a cyclic permutation of a finite set.
\end{thm}

Theorems \ref{thm:structure} and \ref{thm:HH} imply the following result.

\begin{cor}
Let $X$ be a locally compact, locally connected metric space and $f:X\to X$ be a continuous map.   Let $A\subset X$ be a compact, chain transitive attractor.  Then there is a $k$ such that $A$ decomposes 
into $k$ closed and open equivalence classes, $f$ permutes the equivalence classes cyclically, and $f^k$ restricted to each equivalence class is chain mixing.

\end{cor}

It is worth remarking here that on many compact metric spaces, including manifolds, homeomorphisms generically have no chain transitive attractors.  See \cite{AHK,Hur}.



We conclude this section with an example.  

\begin{ex}
Construct a Denjoy map $f:S^1\to S^1$ by starting with an irrational rotation $R_\alpha$ and replacing $\{R_\alpha^n(0)\}_{n\in\Z}$, the orbit of 0, with a sequence of intervals $\{I_n\}$ such that $\sum_{n\in\Z}\operatorname{length}(I_n)=1$ and $I_n$ is mapped into $I_{n+1}$ by an orientation-preserving homeomorphism (see \cite[\S2.8]{R} for details).  

So defined, $f$ is chain mixing, but not topologically transitive.  In fact, the omega limit set of any $x\in S^1$ is the invariant Cantor set $C=S^1\backslash\bigcup_{n\in\Z}\Int(I_n)$.

If we restrict the map to $C$, then $f|_C$ is still chain mixing and $f|_C$ is topologically transitive.  Thus the space $K$ of connected components of $C$ is $C$ itself, while $\tilde K$ is a single point.
\end{ex}

\section{Estimating recurrence rates}
\label{sec:estimates}

Let $f:X\to X$ be a chain transitive map.  We wish to estimate the chain recurrence time $\crt$ and, in the case that $f$ is chain mixing, the chain mixing time $\mt(\delta)$ for a given $\ep$.  Our technique is to model $\ep$-chains with a subshift of finite type, an idea which has been used before in some other contexts (\cite{A,H,KMV,OC}, for example).  The approach we use here is adapted from that in chapter 5 of \cite{A}.  For calculations of $\crt$ and $\mt(\delta)$ for a variety of dynamical systems, see section \ref{sec:examples}.

\begin{prop}
Let $d'$ be the upper box dimension of $X$ (see \cite[\S6]{P} for a definition of upper box dimension).  There exists a constant $C>0$ such that for small enough $\ep$:
\begin{enumerate}
\item if $f:X\to X$ is chain transitive, then $\crt \le C/\ep^{d'}$.
\item if $f:X\to X$ is chain mixing, then $\mt(\delta)\le  C/\ep^{2d'}$.
\end{enumerate}
\end{prop}

\begin{proof}
Let $\{X_i\}_{i=1}^{\nspan(\frac\ep2)}$ be a minimal cardinality collection of balls of radius $\frac\ep2$ whose
 interiors cover $X$.  Observe that for small enough $\ep$, $\nspan(\frac\ep2)+1\le C/\ep^{d'}$ for some constant  $C$, where $d'$ is the upper box dimension of $X$ (\cite[\S6]{P}).  Construct the $\nspan(\frac\ep2)\times \nspan(\frac\ep2)$ transition matrix $A_\ep$ by setting $a_{ij}=1$ if $f(X_j)\cap X_i \ne \emptyset$ and 0 otherwise.  If $(i_0,\dots,i_k)\in\{1,\dots,\nspan(\frac\ep2)\}^{k+1}$ is an allowable word in the subshift of finite type determined by $A_\ep$, then for any $x\in f^{-1}(X_{i_0})$ and any $x_k\in X_{i_k}$, there is an $\ep$-chain $(x, x_0,\dots,x_k)$ with $x_j\in X_{i_j}$ for each $j$.  Conversely (see Lemma 5.1 of \cite{A}), if $(x_0,\dots,x_k)$ is an $\alpha$-chain, where $\alpha$ is less than the Lebesgue number of the cover $\{X_i\}$ and $x_j\in X_{i_j}$, then the word $(i_0,\dots,i_k)$ is an allowable word.  Thus, because $f$ is chain transitive, for any $i$ and $j$, there is an allowable word starting at $i$ and ending at $j$.  By removing any subwords that begin and end at the same symbol, we can assume that the word has length less than or equal to $\nspan(\frac\ep2)$, and so for any two points $x$ and $y$ there is an $\ep$-chain from $x$ to $y$ of length less than or equal to $\nspan(\frac\ep2)+1\le C/\ep^{d'}$ for small enough $\ep$.  This proves the first part of the result.
 
To prove the second part, observe that since $f$ is chain mixing, the matrix $A_\ep$ is eventually positive.  So if $p$ is the smallest integer such that each entry of $A_\ep^p$ is strictly positive, then $\mt(\delta)\le p+1$.   
Since $p\le (\nspan(\frac\ep2)-1)^2+1$  (\cite{HV}), we have that $\mt(\delta) \le  C/\ep^{2d'}$ for small enough $\ep$.

\end{proof}

We also have a lower bound for $\mt(\delta)$ in the case that $f$ is Lipschitz.  We know that the Lipschitz constant $\lip$ is at least 1, since otherwise the Banach fixed point theorem says that $f$ has an attracting fixed point, contradicting the fact that $f$ is chain mixing.

\begin{prop}\label{prop:LipEst} 
Let $f:X\to X$ be chain mixing and have Lipschitz constant $\lip$.  Let $D$ be the diameter of $X$.  Then for $\delta$ sufficiently small,  $\ds\mt(\delta)\ge \log_{\lip}\left(\frac{D(\lip-1)+2\ep}{2\delta(\lip-1)+2\ep}\right)$ if $\lip>1$, and $\ds\mt(\delta)\ge \frac{D-2\delta}{2\ep}$ if $\lip=1$.
\end{prop}

\begin{proof}
  The diameter of the $n$th iterate of a ball of radius $\delta$ under $B_\ep\circ f$ is
\[\diam ((B_\ep\circ f)^n(B_\delta(x)))  \le  \lip^n(2\delta)+\lip^{n-1}(2\ep)+\lip^{n-2}(2\ep)+\dots+2\ep.\]

So, if $\lip>1$, then \[\diam ((B_\ep\circ f)^n(B_\delta(x)))  \le  \lip^n(2\delta)+\frac{1-\lip^n}{1-\lip}(2\ep).\] Thus if $\diam ((B_\ep\circ f)^n(B_\delta(x)))$ is $D$, then $n$ is at least  $\log_{\lip}\left(\frac{D(\lip-1)+2\ep}{2\delta(\lip-1)+2\ep}\right)$.

When $\lip=1$, $\diam ((B_\ep\circ f)^n(B_\delta(x)))  \le 2\delta + 2n\ep$.

\end{proof}

Give the product space $X\times Y$ the sup metric, i.e., $d((x,y),(x',y'))=\max(d_X(x,x'),d_Y(y,y'))$.

\begin{prop}\label{prop:recPowers}
Let $X$ and $Y$ be compact, $f:X\to X$ and $g:Y \to Y$ be chain recurrent, and $k\in\Z^{+}$. Then $f^k:X\to X$ and $f\times g:X\times Y\to X \times Y$ are also chain recurrent, and for all $\ep>0$, $x\in X$, and $y\in Y$,

\begin{enumerate}
\item $\crt((x,y),f\times g))\ge \max (\crt(x,f), \crt(y,g))$,
\item $\crt((x,y),f\times g)) \le \operatorname{lcm} (\crt(x,f), \crt(y,g))$,
\item $\crt(x,f^k)\ge \ds \frac1k \crt(x,f)$, and
\item there exists an $\ep'\le \ep$ such that $\crt(x,f^k)\le  r_{\ep'}(x,f)$.
\end{enumerate}
\end{prop}

\begin{proof}
Observe that $((x_0,y_0),\ldots, (x_n,y_n))$ is an $\ep$-chain for $f\times g$ if and only if $(x_0,\ldots, x_n)$ and $(y_0,\ldots,y_n)$ are $\ep$-chains for $f$ and $g$ respectively.  Statement (1) and the fact that $f\times g$ is chain recurrent if and only if $f$ and $g$ both are follow immediately.

To prove (2), let $(x_0=x,x_1, \ldots, x_{n}=x)$ and $(y_0=y,y_1, \ldots, y_{m}=y)$ be $\ep$-chains.  Then the $\ep$-chain $$(x_0,\ldots, x_{n}=x_0, x_1, \ldots, x_{n}=x_0,\ldots\ldots, x_{n}=x_0)$$ formed by concatenating $(x_0=x,x_1, \ldots, x_{n}=x)$ with itself $\ds\frac{m}{\gcd(n,m)}$ times has length $\operatorname{lcm}(m,n)$, as does the $\ep$-chain $$(y_0,\ldots, y_{m}=y_0, y_1, \ldots, y_{m}=y_0,\ldots\ldots, y_{m}=y_0)$$ formed by concatenating $(y_0=y,y_1, \ldots, y_{m}=y)$ with itself $\ds\frac{n}{\gcd(n,m)}$ times.  Combining the two gives an $\ep$-chain from $(x,y)$ to itself.

To show that $f$ is chain recurrent if $f^k$ is, and to prove (3), it suffices to observe that if $(x_0,\ldots,x_n)$ is an $(\ep,f^k)$-chain, then $$(x_0,f(x_0),\ \ldots, f^{k-1}(x_0),x_1,f(x_1), \ldots\ldots, f^{k-1}(x_{n-1}), x_n)$$ is an $(\ep,f)$-chain of length $nk$.

To see that $f^k$ is chain recurrent if $f$ is, pick $\ep'$ small enough that for any $(\ep', f)$-chain $(x_0,\ldots,x_k)$, $x_k$ is within $\ep$ of $f^k(x_0)$.  (Such an $\ep'$ exists because of the uniform continuity of $f$.)  Thus, if $(x_0,\ldots,x_{nk})$ is an $(\ep',f)$-chain of length $nk$, then $(x_0,x_k,x_{2k},\ldots,x_{nk})$ is an $(\ep,f^k)$-chain of length $n$.  To prove (4), observe that by taking an $(\ep',f)$-chain of length $n$ from $x$ to $x$ and concatenating it with itself $k$ times, we get an $(\ep',f)$-chain of length $nk$, and thus an $(\ep,f^k)$-chain of length $n$.

\end{proof}

\begin{remark}
If we give $X\times Y$ the Euclidean metric, then (2) becomes ``$\crt((x,y),f\times g) \le \operatorname{lcm} (r_{\ep/2}(x,f), r_{\ep/2}(y,g))$.''
\end{remark}

\begin{prop}
Let $X$ and $Y$ be compact, $f:X\to X$ and $g:Y \to Y$ be chain mixing, and $k\in\Z^{+}$.  Then $f^k:X\to X$ and $f\times g:X\times Y\to X \times Y$ are also chain mixing, and for all $\ep>0$,
\begin{enumerate}
\item $\mt(\delta,f\times g)=\max(\mt(\delta,f), \mt(\delta,g))$
\item $\mt(\delta,f^k)\ge \ds \frac1k \mt(\delta,f)$
\item there exists an $\ep'\le \ep$ such that $\mt(\delta,f^k) \le m_{\ep'}(\delta,f)$.
\end{enumerate}
\end{prop}

\begin{proof}
Statement (1) and the fact that $f\times g$ is chain mixing if and only if both $f$ and $g$ are follow from the definition of chain mixing.  The proofs of (2) and (3) and of the fact that $f^k$ is chain mixing are essentially identical to those of the corresponding statements for chain recurrent maps in Proposition \ref{prop:recPowers}.
\end{proof}

Notice that if $f:X\to X$ is chain recurrent and $S\subset X$ is an invariant subset, then $\crt(f|_S)$ may be greater than or less than $\crt(f)$;  while clearly $\crt(x,f|_S)\ge\crt(x,f)$ for $x\in S$, we take the maximum over a smaller set ($S$ instead of $X$) to get $\crt(f|_S)$.  Consider the following example.

\begin{ex}
Let $X=S_{1}\cup S_{2}\cup S_{3}$ and $X'=S_{1}\cup S_{2}$ where $S_{1}$, $S_{2}$, and $S_{3}$ are the circles shown in Figure~\ref{fig:circs}.   Let $f:X\to X$ be a homeomorphism with fixed points at the south poles and motion indicated by the arrows.  Assume that away from the fixed points the motion on $S_{1}$ is much faster than the motion on $S_{2}$ which is much faster than the motion on $S_{3}$. If $S_1$ stays close enough to $S_2$ near their common fixed point, then a point $x$ near the fixed point can return to itself more quickly via an $\ep$-chain passing through $S_1$ than via one that stays in $S_2$, so $\crt(x,f|_{S_2}) \ge \crt(x,f|_{X'})$, and in fact $\crt(f|_{S_2}) > \crt(f|_{X'})$.  On the other hand, points move very slowly on $S_3$, so $\crt(f|_{S_2}) < \crt(f)$.
\end{ex}

\begin{figure}
\psfrag{S1}{$S_{1}$} 
\psfrag{S2}{$S_{2}$} 
\psfrag{S3}{$S_{3}$} 
\includegraphics{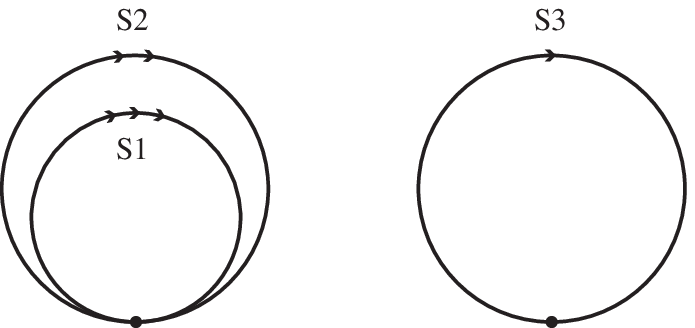}
\caption{}
\label{fig:circs}
\end{figure}

\section{Topological entropy}
\label{sec:entropy}

The following theorem shows that the growth of the chain mixing times can give information about topological entropy.

\begin{thm}\label{thm:entEst}  
Let $f:X\to X$ be chain mixing.  Then the topological entropy, $h(f)$, satisfies
$$\ds h(f)\ge d\cdot \limsup_{\delta\to0}\frac{\log (1/\delta)}{\lim_{\ep\to0}\mt(\delta)},$$ where $d$ is the lower box dimension of $X$.
\end{thm}

\begin{proof}  In \cite{M}, Misiurewicz shows that we can use the growth rate of $\ep$-chains to calculate topological entropy.  Specifically, $$h(f)=\lim_{\delta\to0}\lim_{\ep\to0}\limsup_{n\to\infty} \frac1n\log S(n,\delta,\ep),$$
where $S(n,\delta,\ep)$ is the maximum cardinality of a set of $\delta$-separated $\ep$-chains of length $n$.  (Two $\ep$-chains $(x_0,\dots,x_{n-1})$ and $(y_0,\dots,y_{n-1})$ are $\delta$-separated if $d(x_i,y_i)\ge\delta$ for some $i$.)

For $\alpha>0$, let $N(\alpha)=S(0,\alpha,0)$ (i.e.,  the maximum cardinality of an $\alpha$-separated collection of points).  We claim that for $k\ge0$, $S(k\mt(\delta),\delta,\ep) \ge (N(3\delta))^{k+1}$.  To see this, let $\{x_1,\dots,x_{N(3\delta)}\}$ be a $3\delta$-separated set of points.  By the definition of $\mt(\delta)$, for each of the $(N(3\delta))^{k+1}$ sequences $(i_0,\dots,i_k)$ ($1\le i_j\le N(3\delta)$) there is a sequence of points $(x_{i_1,1},\dots,x_{i_k,k})$ such that for each $j$, $x_{i_j,j}\in B_\delta(x_i)$ and there is an $\ep$-chain of length $\mt(\delta)$ from $x_{i_j,j}$ to $x_{i_{j+1},j+1}$.  Since the points $x_i$ are $3\delta$-separated, the sequences $(x_{i_1,1},\dots,x_{i_k,k})$ are $\delta$-separated.  Thus $S(k\mt(\delta),\delta,\ep) \ge (N(3\delta))^{k+1}$.

Next, let $\nspan(\alpha)$ be the minimum number of balls of radius $\alpha$ necessary to cover $X$.  By the definition of lower box dimension (\cite[\S6]{P}), for small enough $\alpha$, $\nspan(\alpha)\ge C(1/\alpha)^d$ for some positive constant $C$.  Clearly $N(\alpha)\ge \nspan(\alpha)$ (\cite[Lemma 8.1.10]{R}), so $N(\alpha)\ge C(1/\alpha)^d$ as well.

Finally, we have that 
\begin{eqnarray*}h(f) & = & \lim_{\delta\to0}\lim_{\ep\to0}\limsup_{n\to\infty} \frac1n\log S(n,\delta,\ep)
\\ & \ge & \limsup_{\delta\to0}\lim_{\ep\to0}\limsup_{k\to\infty} \frac1{k\mt(\delta)} \log (N(3\delta))^{k+1} 
\\ & = &\limsup_{\delta\to0}\lim_{\ep\to0}\frac{\log N(3\delta)}{\mt(\delta)} 
\\ &\ge & \limsup_{\delta\to0}\lim_{\ep\to0}\frac{\log C/(3\delta)^d}{\mt(\delta)}
\\ & = & d \limsup_{\delta\to0}\frac{\log 1/\delta}{\lim_{\ep\to0}\mt(\delta)},
\end{eqnarray*}
 since $\lim_{\delta\to0}\mt(\delta)=\infty$ (recall that $X$ is not a single point, and it cannot be a finite collection of points since $f$ is chain mixing).
\end{proof}

In the following section we see that the formula in Theorem \ref{thm:entEst} can give the precise value of the topological entropy.











\section{Examples}
\label{sec:examples}

We conclude with three examples---rigid rotations of the circle, subshifts of finite type, and the doubling map on the circle---in which we illustrate the ideas in the previous sections.

\begin{ex}[Rigid rotations]
Consider the dynamical system $R_{\alpha}:S^{1}=\R/\Z\to S^{1}$ which is the rigid rotation by $\alpha\in\R$ given by $R_{\alpha}(x)=\alpha+ x$ (mod 1).  We know that when $\alpha$ is rational every point is periodic and when $\alpha$ is irrational every point has a dense orbit.  Thus, in either case $S^{1}$ is chain mixing.  The chain mixing time is just the least integer greater than or equal to $\frac1{2\ep}$. 

The chain recurrence time for $R_{\alpha}$ depends on the Diophantine properties of $\alpha$.  Let $\ep>0$ be given.  Clearly, $\crt(x,R_{\alpha})=\crt(R_{\alpha})=\crt(R_{\alpha+n})$ for all $x\in S^{1}$ and  $n\in\Z$.  So we may assume that $0<\ep<\alpha<1$ and $x=0$.    Notice that for a rotation, the set of points that can be reached with an $\ep$-chain of length $n$ beginning at $0$ is precisely those points within $n\ep$ of $R_\alpha^{n}(0)$.  Therefore $\crt$ is the smallest integer $n>0$ such that $d(R^{n}_{\alpha}(0),0)<n\ep$.  Equivalently, we want to find the smallest $n>0$ such that \[|n\alpha-m|<n\ep,\text{ or equivalently, }\Big|\alpha-\frac{m}{n}\Big|<\ep\] for some $m\in\Z$.  

If $\alpha$ is rational then $0$ is periodic with some period, $q$.  In this case, for $\ep$ sufficiently small, the shortest $\ep$ chain from $0$ to $0$ is the periodic orbit.  Thus, $\crt=q$.

Now assume $\alpha$ is irrational.  To analyze this case we need to use continued fractions (for details, see \cite{K}).  Let \[[a_{0},a_{1},a_{2},...]=a_{0}+\cfrac{1}{a_{1}+\cfrac{1}{a_{2}+\dotsb}}\] be the continued fraction expansion for $\alpha$, and $p_{k}/q_{k}=[a_{0},a_{1},...,a_{k}]$ the $k$th convergent (assume that $p_{k}/q_{k}$ is in lowest terms).  We say that the rational number $p/q$ ($q>0$) is a \emph{best approximation} of $\alpha$ if the inequalities $p/q\ne r/s$ and $0<s\le q$ imply \[|q\alpha-p|<|s\alpha-r|.\]   In other words, the only way to get a better rational approximation of $\alpha$ than a best approximation is by increasing the denominator.

\begin{thm}[\cite{K}]
\label{thm:best}
Suppose $\alpha$ is irrational.  The rational number $p/q$ ($q>0$) is a best approximation of $\alpha$ if and only if it is a convergent of $\alpha$.
\end{thm}

From Theorem \ref{thm:best} we can conclude that for a given $\ep$, the values of $n$ and $m$ that we wish to find are precisely $q_{k}$ and $p_{k}$ for some convergent $p_{k}/q_{k}$ of $\alpha$.  That is, the function $\crt$ is a nonincreasing piecewise constant function taking on the values $\{q_{k}\}$. Thus, to understand the behavior of $\crt$ as $\ep\to 0$ we must investigate the rate of convergence of $p_{k}/q_{k}$ to $\alpha$.  The following theorem gives a bound on this rate of convergence.

\begin{thm}[\cite{K}]
Let $\{p_{k}/q_{k}\}$ be the sequence of convergents for $\alpha\in\R$.  Then for all $k>1$
\[\displaystyle\Big|\alpha-\frac{p_{k}}{q_{k}}\Big|<\frac{1}{q_{k}^{2}}.\]
\end{thm}

Thus, for a given $\ep$, if $p_{k}/q_{k}$ is the first convergent of $\alpha$ with $q_{k}\ge 1/\sqrt{\ep}$, then  \[\Big|\alpha-\frac{p_{k}}{q_{k}}\Big|<\frac{1}{q_{k}^{2}}\le\ep.\]  In other words, the fastest that $\crt$ can grow is at a rate of $O(1/\sqrt{\ep})$.

In fact, this worst case can be realized.  For example, when $\alpha=(\sqrt{5}-1)/2=[1,1,1,1,\dots]$ we have (\cite{K})  \[\Big|\alpha-\frac{p_{k}}{q_{k}}\Big|=\frac{1}{(\sqrt{5}+\beta_{k})q_{k}^{2}}\] where $\beta_{k}\to 0$ as $k\to\infty$.  Thus, when $\ep=1/q_{k}^{2}$ we have  \[\Big|\alpha-\frac{p_{k}}{q_{k}}\Big|=\frac{\ep}{\sqrt{5}+\beta_{k}}\sim\frac{\ep}{\sqrt{5}}\] for large $k$, so $\crt(R_\alpha)\approx \ds\frac1{\sqrt{\sqrt{5}\ep}}$.  

On the other hand, there are $\alpha$ for which $\crt(R_\alpha)$ grows slowly---as slowly as desired---as $\ep\to 0$.  In particular,  if $\psi:\R^{+}\to\R^{+}$ is any decreasing function with $\psi(x)\to\infty$ as $x\to 0$, then there is an $\alpha$ such that $r_{\ep_{i}}(R_\alpha)\le\psi(\ep_{i})$ for some sequence $\ep_{i}\to0$. To see this we need the following theorem.

\begin{thm}[\cite{K}]
\label{thm:fastconv}
For any function $\varphi:\Z^{+}\to\R^{+}$, there is an irrational number $\alpha$ such that the inequality $|\alpha-p/q|<\varphi(q)$ has an infinite number of solutions in $p$ and $q$ ($q>0$).
\end{thm}

Define $\varphi=\psi^{-1}|_{\Z^{+}}$.  Let $\alpha$ be the irrational number guaranteed by Theorem \ref{thm:fastconv} for this $\varphi$.    Then there exist infinitely many $p$ and $q$ with $|\alpha-p/q|<\varphi(q)$.  Let $\{q_{i}\}$ denote this collection of denominators, listed in increasing order, and let $\ep_{i}=\varphi(q_{i})$.  Then $r_{\ep_{i}}(R_\alpha)\le q_{i}=\varphi^{-1}(\ep_i)=\psi(\ep_{i})$.
\end{ex}

\begin{ex}[Subshifts of finite type]
\label{ex:finitetype}

Let $\sigma:\Sigma\to\Sigma$ be a subshift of finite type given by an $n\times n$ transition matrix $A$, with $k\ge2$, and metric $d(\A,\B)= \ds \sum \frac {\delta(a_i,b_i)}{2^{|i|}}$.  Since an $\ep$-jump can take any element of a cylinder set of diameter $\ep$ to any other element in the same cylinder set, it is easy to see that $\sigma$ is chain transitive if and only if it is topologically transitive.  
(This is because, for a one-sided shift, we can consider only $\ep$-chains that have a jump only at the first step, and afterwards are actual orbits; for a two-sided shift, they can have a jump at the beginning and at the end.)  The shift space $\Sigma$ decomposes into finitely many topologically transitive pieces (\cite[\S\S6.3, 4.4]{LM}), so we assume that $\sigma$ is topologically transitive.  For notational convenience, we consider one-sided shifts (the proofs for two-sided shifts are essentially the same).

Given $\ep>0$ and $\A=(a_0,a_1,\ldots)\in\Sigma$, we want to find $\crt(\A)$.  Let $k$ be the smallest nonnegative integer satisfying $2^{-k-1}<\ep$.  Since $\sigma$ is topologically transitive, the transition matrix $A$ is irreducible (\cite[\S6.3]{LM}), so there is an allowable word $(a_k,a'_1,a'_2,\ldots,a'_{m-1},a_0)$ of length $m\le n$ from $a_k$ to $a_0$.  Thus the sequence $(\B_0=\A,\B_1,\ldots,\B_{k+m}=\A)$ defined by $\B_1=(a_1,a_2,\ldots,a_k,a'_1,a'_2,\ldots,a'_{m-1},a_0,a_1,\ldots)$ and $\B_i=\sigma(\B_{i-1})$ for $i>1$ is an $\ep$-chain from $\A$ to itself of length $k+m$.  Notice that for some $\A\in\Sigma$ we can find an $\ep$-chain from $\A$ to itself that is shorter than this, but for most points in $\Sigma$ this is the shortest possible $\ep$-chain.  Thus, the worst case is that $\crt(\A)=O(\log_2(1/\ep))$.

Since we assume that $\ep<\delta$, a similar argument shows that if $\sigma$ is chain mixing, then $\mt(\delta,\sigma)$ is $O(\log_2(1/\delta))$ if $\sigma$ is one-sided, and $O(\log_2(1/\ep))$ if $\sigma$ is two-sided.  
\end{ex}






\begin{ex}[The doubling map]

Let $f(x)=2x\mod1$ be the doubling map on the circle $S^1=\R/\Z$.  If we consider binary expansions, then $S^1=\{(0.a_1a_2\ldots):a_i\in\{0,1\}\}$, with the identifications $(0.11\ldots)=(0.00\ldots)$, $(0.011\ldots)=(0.100\ldots)$, etc.  Thus the argument in Example \ref{ex:finitetype} shows that $\crt(f) = \lceil\log_2(1/\ep)\rceil$, where $\lceil\cdot\rceil$ is the least integer function.   Arguing as in the proof of Proposition~\ref{prop:LipEst}, we see that  $\mt(\delta,f)=\left\lceil \log_{2}\left(\frac{1+2\ep}{2\delta+2\ep}\right) \right\rceil$.  Then Theorem~\ref{thm:entEst} says that  \[h(f)\ge \limsup_{\delta\to0}\frac{\log (1/\delta)}{\lim_{\ep\to0}\left\lceil \log_{2}\left(\frac{1+2\ep}{2\delta+2\ep}\right) \right\rceil} = \limsup_{\delta\to0}\frac{\log (1/\delta)}{\log_{2}(1/2\delta)}=\log2,\]
so in this case the estimate is exact.

\end{ex}



\end{document}